\documentclass[10pt]{amsart}
\usepackage[english]{babel}
\usepackage{amssymb}
\usepackage{hyperref}
\usepackage[initials]{amsrefs}
\usepackage[all]{xy}
\SelectTips{cm}{}

\newcommand{\bbN}{{\mathbb N}}
\newcommand{\bbQ}{{\mathbb Q}}
\newcommand{\bbR}{{\mathbb R}}

\newcommand{\PP}{{\mathbb P}}

\newcommand{\ZZ}{{\mathbb Z}}
\newcommand{\bbZ}{{\mathbb Z}}
\newcommand{\bbC}{{\mathbb C}}

\newcommand{\calU}{\mathcal{U}}

\newcommand{\bs}{\backslash}
\newcommand{\betti}{\beta^{(2)}}

\newcommand{\res}{\operatorname{res}}
\newcommand{\coind}{\operatorname{coind}}

\newcommand{\coker}{\operatorname{coker}}
\newcommand{\spn}{\operatorname{span}}

\newcommand{\Sym}{\operatorname{Sym}}

\newcommand{\PGL}{\operatorname{PGL}}
\newcommand{\GL}{\operatorname{GL}}
\newcommand{\EL}{\operatorname{EL}}

\newcommand{\SL}{\operatorname{SL}}

\newtheorem{theorem}{Theorem}[section]
\newtheorem{lemma}[theorem]{Lemma}

\newtheorem{corollary}[theorem]{Corollary}

\newtheorem{prop}[theorem]{Proposition}

\theoremstyle{definition}
\newtheorem{definition}[theorem]{Definition}

\newtheorem{example}[theorem]{Example}

\newtheorem{remark}[theorem]{Remark}

\numberwithin{equation}{section}
\begin{document}

\title[$L^2$-cohomology and weak notions of normality]{Weak notions of normality and vanishing up to rank in $L^2$-cohomology}

\author{Uri Bader}
\address{Technion, Haifa, Israel}
\email{uri.bader@gmail.com}

\author{Alex Furman}
\address{University of Illinois at Chicago, Chicago, USA}
\email{furman@math.uic.edu}

\author{Roman Sauer}
\address{Karlsruhe Institute of Technology, Karlsruhe, Germany}
\email{roman.sauer@kit.edu}

\thanks{U.B. and A.F. were supported in part by the BSF grant 2008267.}
\thanks{U.B was supported in part by the ISF grant 704/08.}
\thanks{A.F. was supported in part by the NSF grants DMS 0604611 and 0905977.}


\maketitle
\begin{abstract}
	We study vanishing results for $L^2$-cohomology of countable groups under the presence of subgroups that satisfy some weak normality condition. As a consequence we show that the $L^2$-Betti numbers of $\SL_n(R)$ for any infinite integral domain~$R$ 
	vanish below degree~$n-1$. Another application is the vanishing of all $L^2$-Betti numbers for Thompson's groups $F$ and $T$.
\end{abstract}

\section{Introduction}

An important application of the algebraic theory of $L^2$-Betti 
numbers~\cite{lueck-dimension} (see~\cite{farber} for an alternative approach) is 
that the $L^2$-Betti numbers $\betti_i(\Gamma)$ of a group~$\Gamma$ 
vanish if it has a normal 
subgroup whose $L^2$-Betti numbers vanish. With regard to the \emph{first} $L^2$-Betti 
number one can significantly relax the normality condition to obtain similar 
vanishing results~\cite{peterson+thom}. J.~Peterson and A.~Thom prove in~\cite{peterson+thom} that the first $L^2$-Betti number of a group vanishes if it has a $s$-normal subgroup (defined below) with vanishing first $L^2$-Betti number. 

The 
aim of this article is to extend such vanishing results to  
arbitrary degrees and to present some applications. 
Next we describe the main notions and 
results in greater detail. 

We denote the $\gamma$-conjugate $\gamma^{-1}\Lambda\gamma$ of a subgroup $\Lambda<\Gamma$ 
by $\Lambda^\gamma$. Unless
stated otherwise, all groups are discrete and countable, and all modules are left modules.

\begin{definition}\label{def: weak normality}
	A subgroup $\Lambda$ of a group $\Gamma$ is called
	\begin{enumerate}
		\item\emph{$n$-step s-normal} if for every $(n+1)$-tuple $\omega=(\gamma_0,\ldots, \gamma_n)\in\Gamma^{n+1}$ the
		intersection
		\[\Lambda^\omega:=\Lambda^{\gamma_0}\cap\ldots\cap\Lambda^{\gamma_n}\]
		is infinite.
		\item\emph{s-normal} if it is $1$-step s-normal.
	\end{enumerate}
\end{definition}

\begin{example}
	The subgroup of upper triangular matrices 
	\[
		\begin{pmatrix}
			* & *\\
			0 & *
		\end{pmatrix}<\SL_2(\bbZ[1/p])
	\]
	inside $\SL_2(\bbZ[1/p])$ is $1$-step s-normal but not $2$-step s-normal. The fact that 
	it is $s$-normal can be 
	verified directly or is a special case of the more general results 
	in Subsection~\ref{subsec: Sl and El}. The fact that it is not $2$-step s-normal can 
	again be verified directly; it is also a consequence of 
	Corollary~\ref{cor: vanishing for n-step s-normal amenable subgroups} below and 
	the non-vanishing 
	\begin{equation}\label{eq: second betti}
		\betti_2(\SL_2(\bbZ[1/p]))\ne 0
	\end{equation}
	of the second $L^2$-Betti number of $\SL_2(\bbZ[1/p])$. The group 
	$\SL_2(\bbZ[1/p])$ is an irreducible lattice in $\SL_2(\bbR)\times \SL_2(\bbQ_p)$. 
	The latter locally compact group contains a product of non-abelian free groups as a (reducible) 
	lattice. Hence $\SL_2(\bbZ[1/p])$ is measure equivalent to a product of non-abelian free 
	groups. By an important theorem of D.~Gaboriau~\cite{gaboriau} the non-vanishing 
	of $\betti_n$ is an invariant under measure equivalence, and the second $L^2$-Betti 
	number of 
	a product of non-abelian free groups is non-zero by the Kuenneth formula for $L^2$-cohomology. 
\end{example}

The following is our main result. Recall that the zeroth $L^2$-Betti number of 
a group is zero if and only if the group is infinite. 

\begin{theorem}\label{thm: vanishing for groups with amenable normal subgroups}
Let $\Lambda<\Gamma$ be a subgroup. 
Assume that 
\[
	\betti_i(\Lambda^\omega)=0. 
\]
for all integers $i,k\ge 0$ with $i+k\le n$ and every 
$\omega\in\Gamma^{k+1}$. In particular, $\Lambda$ is an $n$-step s-normal subgroup of  $\Gamma$. 
Then 
\[\betti_i(\Gamma)=0 \text{ for every $i\in\{0,\ldots, n\}$.}\]
\end{theorem}

Recall that $\Lambda<\Gamma$ is called \emph{commensurated} if 
$\Lambda\cap\Lambda^\gamma$ is of finite index in $\Lambda$ and $\Lambda^\gamma$ for 
every $\gamma\in\Gamma$. The corollary follows from the preceding theorem and the fact 
that one has the relation 
$\betti_i(\Gamma')=[\Gamma:\Gamma']\cdot\betti_i(\Gamma)$ for a subgroup 
$\Gamma'<\Gamma$ of finite index. 

\begin{corollary}\label{cor: vanishing for almost normal subgroups}
	Let $\Lambda<\Gamma$ be a commensurated subgroup. If $\betti_i(\Lambda)=0$ for every $i\in\{0,\ldots, n\}$, then also
	$\betti_i(\Gamma)=0$ for every $i\in\{0,\ldots, n\}$.
\end{corollary}

The theorem above implies together with the vanishing of $L^2$-Betti numbers 
of infinite amenable groups~\cite{cheeger+gromov} the following. 

\begin{corollary}\label{cor: vanishing for n-step s-normal amenable subgroups}
	Let $\Lambda<\Gamma$ be an $n$-step s-normal and amenable subgroup. Then the $L^2$-Betti numbers of $\Gamma$ vanish up to degree~$n$, that is, $\betti_i(\Gamma)=0$ for
	every $i\in\{0,\ldots, n\}$.
\end{corollary}

By taking a suitable subgroup $\Lambda$ inside the special 
linear group $\Gamma=\SL_n(R)$ over a ring $R$, 
Theorem~\ref{thm: vanishing for groups with amenable normal subgroups} yields 
the following application (proved in Subsection~\ref{subsec: Sl and El}). 

\begin{theorem}\label{thm: application to SL}
	Let $R$ be an infinite integral domain. Let $n\ge 2$. Then 
	\[
		\betti_i(\SL_n(R))=0\text{ for every $i\in\{0,\ldots,n-2\}$.} 
	\]
\end{theorem}

In addition, we have a statement about degree $n-1$: 

\begin{theorem}\label{thm: extended application to SL}
Assume that a ring $R$ satisfies at least one of the following properties: 
\begin{enumerate}
\item $R$ is an infinite field.
\item
$R$ is a subring of the field $F(t)$ of rational functions over a finite field $F$, and 
$R$ contains an invertible element $\alpha$ that is not a root of unity.
\item $R$ is a subring of the field $\bar{\bbQ}$ of algebraic numbers, and 
$R$ contains an invertible element $\alpha$ that is not a root of unity.
\end{enumerate}
Then one has 
	\[
		\betti_{n-1}(\SL_n(R))=0.
	\]
\end{theorem}

If $\SL_n(R)$ is a lattice in a semisimple Lie group, e.g. in the case $R=\bbZ$ or, 
more generally, $R$ being a subring of algebraic \emph{integers},  
then much more is known than in the preceding theorems. It follows from results 
of Borel, which rely on global analysis on the associated symmetric space, that the $L^2$-Betti numbers vanish except possibly in the middle dimension 
of the symmetric space~\cite{borel, olbrich}. However, 
the interesting and new case of the preceding theorems is the one where 
$\SL_n(R)$ is \emph{not} a lattice in a semisimple Lie group; take e.g. $R=\bbZ[x_1,x_2,\ldots, x_d]$. According to results of 
Y.~Shalom~\cite{shalom} and L.~Vaserstein the so-called 
\emph{universal lattice} $\SL_n(\bbZ[x_1,\ldots, x_d])$ has property (T) provided 
that $n\ge 3$; M.~Mimura~\cite{mimura} showed that for $n\ge 4$ the universal lattice has even property $F_{L^p}$, $p\in (1,\infty)$, as defined by Bader-Furman-Gelander-Monod. 
M.~Ershov and A.~Jaikin-Zapirain showed property $(T)$ for the 
groups $\EL_n(\bbZ\langle x_1,\ldots, x_d\rangle)$, $n\ge 3$, of \emph{noncommutative 
universal lattices}~\cite{ershov}. 

Of course, property (T) implies the vanishing of the first $L^2$-Betti number, but nothing was known before about the $L^2$-Betti numbers of universal lattices in 
higher degrees.

The following application of Theorem~\ref{thm: vanishing for groups with amenable normal subgroups} (proved in Subsection~\ref{subsec: Thompson}) was kindly pointed to us by Nicolas Monod remarking on an earlier draft of this paper.

\begin{theorem} \label{Thompson}
All $L^2$-Betti numbers of Thompson's groups $F$ and $T$ vanish.
\end{theorem}

The groups $F$ and $T$ were invented by R.~Thompson in~ 1965. 
In unpublished work Thompson proved that the 
group $T$ is a finitely presented, infinite, simple group. 
The vanishing of $L^2$-Betti numbers for Thompson's group~$F$ was proved before in a 
different way 
by L\"uck~\cite{lueck-book}*{Theorem~7.10. on p.~298}.

\begin{remark}\label{rem: amenable radical}
	In a forthcoming paper~\cite{envelopes} we show that, if a locally compact group $G$ has 
	a non-compact amenable radical, then every lattice of $G$ has an infinite amenable commensurated subgroup. In particular, every lattice of $G$ has vanishing 
	$L^2$-Betti numbers by a theorem of Cheeger-Gromov~\cite{cheeger+gromov} and 
	Corollary~\ref{cor: vanishing for almost normal subgroups}.
\end{remark}

\begin{example}
	The subgroup $\bbZ\cong\langle x\rangle$ of the Baumslag-Solitar group 
		\[
			BS(p,q)=\langle x,t~|~tx^pt^{-1}=x^q\rangle
		\]
	is commensurated but not normal. 
	Corollary~\ref{cor: vanishing for almost normal subgroups} 
	yields that the $L^2$-Betti numbers of $BS(p,q)$ vanish. This result is part of  
	earlier work of W.~Dicks and P.~Linnell~\cite{dicks+linnell} about 
	$L^2$-Betti numbers of one-relator groups. 
\end{example}

{\bf Acknowledgement:} We warmly thank Nicolas Monod for pointing to us that Theorem~\ref{Thompson} follows from Theorem~\ref{thm: vanishing for groups with amenable normal subgroups}.
R.S. wishes to thanks the Mittag-Leffler institute and U.B wishes to thank the math department of the University of Orleans for their hospitality while the final work on this paper was done.

\section{$L^2$-cohomology} 
\label{sec:review}

Our background reference for $L^2$-Betti numbers is~\cite{lueck-book}. 
$L^2$-Betti numbers have various definitions with different levels of generality. A modern and algebraic description that applies to arbitrary groups was
given by W.~L\"uck~\cite{lueck-dimension}. He introduced a dimension function for arbitrary modules over the group von
Neumann algebra $L(\Gamma)$ and showed that the $i$-th $L^2$-Betti number $\betti_i(\Gamma)$ in the sense of
Cheeger-Gromov~\cite{cheeger+gromov} can be expressed
as
\[
	\betti_i(\Gamma)=\dim_{L(\Gamma)}H_i(\Gamma, L(\Gamma)).
\]
The dimension function $\dim_{L(\Gamma)}$ extends to a dimension
function $\dim_{\calU(\Gamma)}$ for modules over
the algebra $\calU(\Gamma)$ of densely defined, closed operators affiliated to $L(\Gamma)$ in the sense that
\[
	\dim_{{L(\Gamma)}}(M)=\dim_{\calU(\Gamma)}\bigl(\calU(\Gamma)\otimes_{{L(\Gamma)}} M\bigr)
\]
for every ${L(\Gamma)}$-module~$M$~\cite{reich}*{Proposition~3.8}. 
 One has~\cite{reich}*{Proposition~5.1}
\begin{equation}\label{eq: betti homology}
	\betti_i(\Gamma)=\dim_{\calU(\Gamma)}H_i\bigl(\Gamma, \calU(\Gamma)\bigr).
\end{equation}
We refer to~~\cite{lueck-book}*{Chapter~8} for more information about this way of defining $L^2$-Betti numbers.
The algebra $\calU(\Gamma)$ of affiliated operators is a \emph{self-injective} ring, that is, the functor $M\mapsto \hom_{\calU(\Gamma)}(M, \calU(\Gamma))$
is exact~\cite{berberian}. A.~Thom firstly exploited this property for the computation of
$L^2$-invariants~\cite{thom}. Later we need the following lemma. 

\begin{lemma}\label{ref:l2 induction}
	Let $\Lambda<\Gamma$ be a subgroup. If $\betti_i(\Lambda)=0$, then 
	\[
		H^i(\Lambda, \calU(\Gamma))=0. 
	\]
\end{lemma}
\begin{proof}
	The ring $\calU(\Lambda)$ is von Neumann regular~\cite{lueck-book}*{Theorem~8.22 on p.~327}. Thus $\calU(\Gamma)$ is a flat $\calU(\Lambda)$-module~\cite{lueck-book}*{Lemma~8.18 on p.~326}. So we have 
	\[
		H_i(\Lambda, \calU(\Gamma))\cong \calU(\Gamma)\otimes_{\calU(\Lambda)} H_i(\Lambda, \calU(\Lambda)). 
	\]
	Uniqueness of $\dim_{\calU(\Lambda)}$-dimension~\cite{reich}*{Theorem~3.11} 
	and flatness of the functor $\calU(\Gamma)\otimes_{\calU(\Lambda)}\_$ yield that for any 
	$\calU(\Lambda)$-module $M$ we have 
	\[
	\dim_{\calU(\Gamma)}\bigl(\calU(\Gamma)\otimes_{\calU(\Lambda)}M\bigr)=\dim_{\calU(\Lambda)}(M).
	\]
	In particular, it follows that 
	\[
		\dim_{\calU(\Gamma)}\bigl(H_i(\Lambda, \calU(\Gamma))\bigr)=
		\dim_{\calU(\Lambda)}\bigl(H_i(\Lambda, \calU(\Lambda))\bigr)\overset{\eqref{eq: betti homology}}{=}\betti_i(\Lambda)=0. 
	\]
	By~\cite{thom}*{Corollary~3.3} this yields that 
	\[
		\hom_{\calU(\Gamma)}\bigl(H_i(\Lambda, \calU(\Gamma)), \calU(\Gamma)\bigr)=0.
	\]
	Since $\calU(\Gamma)$ is self-injective, as mentioned above, the latter module is 
	isomorphic to $H^i(\Lambda, \calU(\Gamma))$. 
\end{proof}

\section{Proof of Theorem~\ref{thm: vanishing for groups with amenable normal subgroups}}
\label{sec: proof of main result}

For a $\bbC\Gamma$-module $M$, we use the notation 
\[
	M^\Gamma=\{m\in M~|~\gamma m=m\text{ for every $\gamma\in\Gamma$}\}. 
\]
For a subgroup $\Lambda<\Gamma$ and a $\bbC\Lambda$-module $M$, the $\bbC\Gamma$-module 
\[
	\coind_\Lambda^\Gamma(M):=\hom_{\bbC\Lambda}(\bbC\Gamma, M)
\]
given by the $\Gamma$-action 
\[
	(\gamma_0 f)(x)=f(x\gamma_0)\text{ for $f\in\coind_\Lambda^\Gamma(M)$ and $\gamma\in\Gamma$}
\]
is called the \emph{co-induced} $\bbC\Gamma$-module~\cite{brown-book}*{III.5}. For a 
$\bbC\Gamma$-module $N$ we denote the \emph{restriction} of $N$ to a $\bbC\Lambda$-module 
by $\res_\Lambda^\Gamma(N)$. We use the notation for the restriction only for 
emphasis; we often drop the $\res_\Lambda^\Gamma$-notation. 

\subsection{A sequence of modules for dimension-shifting} 
\label{sub:a_sequence_of_modules_for_dimension_shifting}

In the sequel let $\Gamma$ be a group, $\Lambda<\Gamma$ a subgroup, and let 
$$M_0=\calU(\Gamma)$$ 
regarded as a $\bbC\Gamma$-module. 
Starting with $M_0$, 
consider the following inductively defined sequence of $\bbC\Gamma$-modules,
whose study is motivated by the use of dimension-shifting in the proof of
Theorem~\ref{thm: vanishing for groups with amenable normal subgroups}.
\begin{equation}\label{eq: modules for dim shifting}
	M_{i+1}:=\coker\bigl(M_i\to\coind_\Lambda^\Gamma(\res_\Lambda^\Gamma(M_i))\bigr).
\end{equation}

The homomorphism $M_i\to\hom_{\bbC\Lambda}(\bbC\Gamma, M_i)$ for the cokernel is
$m\mapsto (\gamma\mapsto \gamma m)$; it is $\bbC\Gamma$-equivariant. So this
declares inductively the $\bbC\Gamma$-module structure on $M_i$.

\begin{lemma}\label{lem:algebraically mixing}
      Assume that for all integers $j,k\ge 0$ with $j+k\le n$ and for every 
	  $\omega\in\Gamma^{k+1}$ one has 
	  \[
	  	\betti_j(\Lambda^\omega)=0. 
	  \]
	  Then for all integers $i,j,k\ge 0$ with $i+j+k\le n$ and for every 
	  $\omega\in\Gamma^{k+1}$ one has 
	  \begin{align}
	  		H^j(\Lambda^\omega, \res_{\Lambda^\omega}^\Gamma(M_i))&=0,\label{eq:M_i formula}\\
	        H^j(\Lambda^\omega, \res_{\Lambda^\omega}^\Gamma(\coind_\Lambda^\Gamma(M_{i-1})))&=0\text{ if $i\ge 1$.}\label{eq:coind formula}
	  \end{align}
\end{lemma}

\begin{proof}
We run an induction over $i\ge 0$. By Lemma~\ref{ref:l2 induction} 
the basis $i=0$ is equivalent to our assumption. Assume the statement is true for a fixed $i\ge 0$ and all $j,k\ge 0$ with 
$i+j+k\le n$. We show that the assertion holds for $i+1$ and all $j,k\ge 0$ 
with $i+1+j+k\le n$: 

The short exact sequence of $\bbC\Lambda^\omega$-modules 
\[
	0\to \res_{\Lambda^\omega}^\Gamma(M_i)\to \res_{\Lambda^\omega}^\Gamma(\coind_\Lambda^\Gamma(M_i))\to \res_{\Lambda^\omega}^\Gamma(M_{i+1})\to 0
\]
induces a long exact sequence in cohomology for which we consider the following part: 
\begin{multline*}
	\ldots\to H^j\bigl(\Lambda^\omega,  \res_{\Lambda^\omega}^\Gamma(\coind_\Lambda^\Gamma(M_i))\bigr)\to H^j\bigl(\Lambda^\omega, \res_{\Lambda^\omega}^\Gamma(M_{i+1})\bigr)\to\\\to H^{j+1}\bigl(\Lambda^\omega, \res_{\Lambda^\omega}^\Gamma(M_i)\bigr)\to\ldots
\end{multline*}
The homology group on the right vanishes by induction hypothesis. It remains to show 
that the homology group on the left vanishes. Mackey's double coset formula~\cite{brown-book}*{III.5} says that after a choice of 
a set $E$ of representatives of the double coset space $\Lambda^\omega\bs\Gamma/\Lambda$ we 
obtain an isomorphism of $\bbC\Lambda^\omega$-modules: 
\[
\res_{\Lambda^\omega}^\Gamma(\coind_\Lambda^\Gamma(M_i))\cong \prod_{\gamma\in E}\coind_{\Lambda^\omega\cap\Lambda^{\gamma^{-1}}}^{\Lambda^\omega}(\res_{\Lambda^\omega\cap\Lambda^{\gamma^{-1}}}^{\Lambda^{\gamma^{-1}}}(\gamma M_i)) 	
\]
Applying the Shapiro lemma and the induction hypothesis yields 
\begin{align*}
	H^j\bigl(\Lambda^\omega,  \res_{\Lambda^\omega}^\Gamma(\coind_\Lambda^\Gamma(M_i))\bigr) &=
	    \prod_{\gamma\in E} H^j\bigl(\Lambda^\omega, \coind_{\Lambda^\omega\cap\Lambda^{\gamma^{-1}}}^{\Lambda^\omega}(\res_{\Lambda^\omega\cap\Lambda^{\gamma^{-1}}}^{\Lambda^{\gamma^{-1}}}(\gamma M_i))\bigr)\\
	&= \prod_{\gamma\in E}H^j\bigl(\Lambda^\omega\cap\Lambda^{\gamma^{-1}}, \res_{\Lambda^\omega\cap\Lambda^{\gamma^{-1}}}^{\Lambda^{\gamma^{-1}}}(\gamma M_i)\bigr)\\
	&= \prod_{\gamma\in E}H^j\bigl(\gamma^{-1}\Lambda^\omega\gamma\cap\Lambda, \res_{\gamma^{-1}\Lambda^\omega\gamma\cap\Lambda}^{\Lambda}(M_i)\bigr)\\
	&= 0. \qedhere
\end{align*}
\end{proof}

\subsection{Conclusion of proof of Theorem~\ref{thm: vanishing for groups with amenable normal subgroups}} 
    Retain the setting of 
    Theorem~\ref{thm: vanishing for groups with amenable normal subgroups}. 
    It suffices 
	to verify that the restriction homomorphism 
	\[\res\colon H^i(\Gamma, M_0)\to H^i(\Lambda, M_0)\]
	is injective for every $i\in\{1,\ldots, n\}$. We employ the
	technique of dimension-shifting~\cite{brown-book}*{III.7}: 
	
	For $i,j\ge 0$ with $i+j\le n$ the Shapiro lemma and~\eqref{eq:M_i formula} yield 
	that 
	\[
		H^j\bigl(\Gamma, \coind_\Lambda^\Gamma(M_i)\bigr)\cong H^j(\Lambda, M_i)=0.
	\]
	From the long exact sequence 
	\begin{multline*}
		\ldots\to H^j\bigl(\Gamma, \coind_\Lambda^\Gamma(M_i)\bigr)\to 
		H^j\bigl(\Gamma, M_{i+1}\bigr)\xrightarrow{\partial} \\
	    H^{j+1}\bigl(\Gamma, M_i\bigr)\to H^{j+1}\bigl(\Gamma, \coind_\Lambda^\Gamma(M_i)\bigr)\to\ldots
	\end{multline*}
	one obtains, for any $i\in\{0,\ldots, n\}$, natural 
	isomorphisms 
	\[
		H^i\bigl(\Gamma, M_0\bigr)\xleftarrow{\cong} H^{i-1}\bigl(\Gamma, M_1\bigr)\xleftarrow{\cong}\ldots\xleftarrow{\cong} H^1\bigl(\Gamma, M_{i-1}\bigr)\xleftarrow{\cong} H^0\bigl(\Gamma, M_i\bigr).
	\]
	Using~\ref{eq:coind formula}, we argue similarly to see that there is 
	a sequence of injective homomorphisms 
	\[
		H^i\bigl(\Lambda, M_0\bigr)\hookleftarrow H^{i-1}\bigl(\Lambda, M_1\bigr)\hookleftarrow\ldots\hookleftarrow H^1\bigl(\Lambda, M_{i-1}\bigr)\hookleftarrow H^0\bigl(\Lambda, M_i\bigr).
	\]
    for any $i\in\{0,\ldots, n\}$. 
	In particular, we obtain, for $i\in\{0,\ldots, n\}$, the following
	commutative square with an upper horizontal isomorphism and a lower horizontal 
	monomorphism:
	\[
		\xymatrix{
		H^0(\Gamma, M_i)\ar[d]^{\res}\ar[r]^\cong &H^i(\Gamma, M_0)\ar[d]^{\res}\\
		H^0(\Lambda, M_i)\ar@{^(->}[r] &H^i(\Lambda,M_0)
		}
	\]
So it is enough to show that the
left restriction map is injective. Since it is given by the inclusion $M_i^\Gamma\hookrightarrow M_i^\Lambda$, this is obvious. 

\section{Applications}

\subsection{The groups $SL_n$ and $EL_n$ over general rings}\label{subsec: Sl and El}

The subgroup of $\GL_n(R)$ that is generated by elementary matrices is 
denoted by $\EL_n(R)$. 
 
\begin{theorem} \label{main-application}
Let $R$ be an infinite integral domain, $K$ its field of fractions. For some $n\geq 2$, let $\Gamma<\GL_n(K)$ be a countable group which contains a finite index subgroup of 
$\EL_n(R)$. 

Then there exists a subgroup $\Lambda<\Gamma$ such that
for every $k<n$ and every $\omega\in \Gamma^{k}$, $\Lambda^\omega$ contains an infinite amenable normal subgroup.

Assume in addition that $\Gamma$ contains a finite index subgroup of $\SL_n(R)$ and 
for every ideal $\{0\}\neq I\lhd R$ there exist
infinitely many invertible elements $x\in R$ such that $x^n-1 \in I$. 
Then also for every $\omega\in \Gamma^{n}$, $\Lambda^\omega$ contains an infinite amenable normal subgroup.
\end{theorem}

\begin{proof}
We let $V=K^n$, and $e_1,\ldots, e_n$ be the standard basis.
We denote by $Q<\GL_n(V)$ the stabilizer of the line
$V_1=\spn\{e_1\} \in \PP(V)$ and
$S\triangleleft Q$ be the kernel of the obvious homomorphism
$Q \to \PGL(V/V_1)$.
Clearly, $S$ is two-step solvable, thus amenable.
We let $V_2=\spn_K\{e_2,\dots, e_n\}$,
thus $V=V_1\oplus V_2$.

Let $\Lambda=\Gamma\cap Q$.
For given $k$ and $\omega=(\gamma_0,\ldots,\gamma_{k-1})\in \Gamma^{k}$
we consider the group $\Lambda^\omega$.
Examining whether it contains an infinite amenable normal subgroup,
we may and will assume that $\gamma_0=e$.
For $i\in\{1,\ldots,k-1\}$ we let $t_i\in K$ and $u_i\in V_2$
be defined by 
\[\gamma_i^{-1}e_1=u_i+t_i e_1.\]
We set
\( U=\spn\{u_1,\dots,u_{k-1}\} < V_2 \).

Assume $U\lneq V_2$.
Then there exists a nontrivial functional $\phi\in V_2^*$ which vanishes on $U$.
Multiplying $\phi$ by the common denominator of $\phi(e_2),\ldots,\phi(e_n)\in K$,
we may assume that $\{\phi(e_2),\ldots,\phi(e_n)\}\subset R$.
For $r\in R$ we define $T_r:V\to V$ by 
\[T_r(v)=v+r\phi\circ p_2(v)\cdot e_1, \]
where $p_2:V\to V_2$ is the projection. 
Observe that $r\mapsto T_r$ is an injection of the additive group of $R$ into $\EL_n(R)$,
whose image (up to a finite index) is in $\Lambda^\omega\cap S$.
We deduce that $\Lambda^\omega\cap S$ is infinite.
This is an infinite amenable normal subgroup of $\Lambda^\omega$, as required.

If $k<n$, looking at the dimensions yields that $U\lneq V_2$,
thus proving the first part of the theorem. 

We now consider the case $k=n$.
We assume further that $\Gamma$ contains $\SL_n(R)$ up to finite index 
and that for every ideal $\{0\}\neq I\lhd R$ there exist
infinitely many invertible elements $x\in R$ such that $x^n-1 \in I$.
Again we will show that the amenable normal subgroup $\Lambda^\omega \cap S$ is infinite.

By the argument above it remains to deal with the case $U=V_2$. Hence we will assume that $U=V_2$, thus $\{u_1,\ldots,u_{n-1}\}$ forms a basis of $V_2$.
We let $\psi\in V_2^*$ be the functional defined by $\psi(u_i)=t_i$.
We let $r\in R\backslash\{0\}$ be such that $\{r\psi(e_2),\ldots,r\psi(e_n)\}\subset R$,
and we set $I=(r)$ to be the ideal generated by~$r$.
Fixing an invertible element $x\in R$ such that $x^n-1 \in I$,
and letting $q_x\in R$ be an element satisfying
$x^{-(n-1)}-x=q_xr$,
we define $S_x:V\to V$ by setting for $t\in K$ and $u\in V_2$
\[ S_x(te_1+u)= (x^{-(n-1)}t-q_xr\psi(u))e_1 + xu. \]
It is clear that, for every such $x$, $S_x$ is in $\SL_n(R)\cap S$ and stabilizes $\gamma_iV_1$ for every $i=0,\ldots n-1$,
thus $\Lambda^\omega \cap S$ is infinite.
\end{proof}

Our next goal will be to show that some integral domains satisfy the condition appearing in the
previous theorem.

\begin{prop} \label{algebraic-numbers}
	Assume that a ring $R$ satisfies at least one of the following properties: 
	\begin{enumerate}
	\item $R$ is an infinite field.
	\item
	$R$ is a subring of the field $F(t)$ of rational functions over a finite field $F$, and 
	$R$ contains an invertible element $\alpha$ that is not a root of unity.
	\item $R$ is a subring of the field $\bar{\bbQ}$ of algebraic numbers, and 
	$R$ contains an invertible element $\alpha$ that is not a root of unity.
	\end{enumerate}
Then for every $n\in\bbN$ and for every ideal $\{0\}\neq I\lhd R$ there exist
infinitely many invertible elements $x\in R$ such that $x^n-1 \in I$.
\end{prop}

The proof of the proposition in case $R$ is a ring of algebraic numbers will depend on the following elementary lemma.

\begin{lemma} \label{elementry-lemma}
Given $\alpha\in \bar{\bbQ}$ and $0\neq k\in \bbN$,
the ring $\ZZ[\alpha]/(k)$ is finite.
\end{lemma}

\begin{proof}[Proof of \ref{elementry-lemma}]
By the general version of the Chinese remainder theorem (for the ring $\bbZ[\alpha]$),
for coprime $k_1,k_2\in\bbN$, the two ideals $(k_1k_2)$ and $(k_1)\cap(k_2)$ coincide
and $\bbZ[\alpha]/(k_1k_2)\simeq \bbZ[\alpha]/(k_1)\times \bbZ[\alpha]/(k_2)$.
It follows that we may assume that $k=p^j$ is a prime power.
We now prove the statement that $\ZZ[\alpha]/(p^j)$ is finite by induction on $j$.
For $j=1$ the statement is clear, as this is a finite dimensional vector space over $\ZZ/(p)$.
For the induction step, observe that the statement is equivalent to the statement that
in $\ZZ[\alpha]$, for some $i$, $\alpha^i-1$ is in the ideal $(p^j)$.
For this statement induction applies easily:
$\alpha^i = 1+ p^jr$ implies $\alpha^{ip}=(1+p^jr)^p=1+p^{j+1}r'$.
\end{proof}

\begin{proof}[Proof of \ref{algebraic-numbers}]
The case that $R$ is an infinite field is trivial.

Assume $R<F(t)$ and that $\alpha\in R$ is an invertible element which is not a root of unity.
We assume (as we may upon replacing $F$ by $F\cap R$) that $R$ is an $F$-algebra,
thus $F[\alpha,\alpha^{-1}]<R$.
Let $\{0\}\neq I\lhd R$ be given.
We claim that the image of $\alpha$ in $(R/I)^\times$ is torsion.
We first observe that $I\cap F[\alpha,\alpha^{-1}] \neq \{0\}$.
Indeed, $F(t)$ is a finite field extension of $F(\alpha)$ (it is finitely generated and of transcendental degree 0),
so if $\sum a_i\beta^i$ is a minimal polynomial over $F(\alpha)$ for some non-zero function $\beta\in I$ with $a_i\in F[\alpha]$
then $a_0\in I$.
The claim follows from the obvious fact that $F[\alpha,\alpha^{-1}]/(I\cap F[\alpha,\alpha^{-1}])$
is a finite extension of $F$, hence finite.
Now, if $\alpha^m-1 \in I$, then the set $\{\alpha^{jm}~|~j\in \bbZ\}$ contains, for every $n$, infinitely many invertible elements $x$ with $x^n-1\in I$.

Assume now that $R<\bar{\bbQ}$.
Again, we claim that the image of $\alpha$ in $(R/I)^\times$ is torsion, for any given $\{0\}\neq I\lhd R$.
We first observe that $I\cap\bbZ\neq \{0\}$.
Indeed, if $\sum a_i\beta^i$ is a minimal polynomial for some non-zero algebraic number $\beta\in I$ with $a_i\in\bbZ$
then $a_0\in I$.
Thus, in order to prove the claim it is enough to show that the image of $\alpha$ is torsion in $(R/(k))^\times$ for every $k\in \bbN$.
This follows from Lemma~\ref{elementry-lemma}.
As before, if $\alpha^m-1 \in I$, then the set $\{\alpha^{jm}~|~j\in \bbZ\}$ contains, for every $n$, infinitely many invertible elements $x$ with $x^n-1\in I$.
\end{proof}

\begin{proof}[Proofs of Theorems~\ref{thm: application to SL} and~\ref{thm: extended application to SL}]
	By a theorem of Cheeger and Gromov~\cite{cheeger+gromov} all $L^2$-Betti numbers
	of a group vanish if the group has an infinite normal amenable subgroup.
	Hence Theorem~\ref{thm: vanishing for groups with amenable normal subgroups}
	and the first part of Theorem~\ref{main-application} yield
	Theorem~\ref{thm: application to SL}. Similarly and
	using Proposition~\ref{algebraic-numbers} in addition, one obtains
	Theorem~\ref{thm: extended application to SL}.
\end{proof}

\subsection{Thompson's groups}\label{subsec: Thompson}
Thompson's group $T$ is defined as the group of piecewise linear homeomorphisms of the circle $\bbR/\bbZ$ that are differentiable except at 
finitely many dyadic rational numbers, i.e.~points in $\bbZ[\frac{1}{2}]/\bbZ$, 
and such that the slopes on intervalls of differentiability are powers of~$2$ with 
respect to the obvious flat structure on $\bbR/\bbZ$. 
Thompson's groups $F$ is defined to be the stabilizer of $0\in \bbR/\bbZ$ in $T$.

\begin{proof}[Proof of Theorem~\ref{Thompson}]
Let $n\ge 1$. 
Let $\Lambda\subset F$ be the stabilizer subgroup inside $F$ of a finite set of $(n+1)$-many dyadic rational points. 
For any $\omega\in F^m$ with $m\ge 1$, the subgroup $\Lambda^\omega\subset F$ is 
the stabilizer subgroup of a finite set of $d$ dyadic rational points with some $d\in \{n+1, \ldots, m(n+1)\}$. By the 
description above it is evident that $\Lambda^\omega\cong F^d$. 
The $L^2$-Betti numbers of any $d$-fold product of 
infinite groups, thus of $\Lambda^\omega$, vanish up to degree $d-1\ge n$ 
by repeated application of the Kuenneth formula in $L^2$-cohomology~\cite{lueck-book}*{Theorem~6.54. on p.~265}. 
Now Theorem~\ref{thm: vanishing for groups with amenable normal subgroups} implies 
that the $L^2$-Betti numbers of $F$ vanish up to degree $n$, and since $n$ was arbitrary, 
Theorem~\ref{Thompson} for the group $F$ is proved. 
For the group $T$ we run the almost the same argument, taking $\Lambda$ to be the 
stabilizer inside $T$ and considering $\omega\in T^m$. We again obtain that  $\Lambda^\omega\cong F^d$ and finish the argument as above. 
\end{proof}

\subsection{Permutation group theoretic criterion}

\begin{theorem}
Let $\Gamma$ be a countable group, $\Lambda<\Gamma$ an amenable subgroup such that the closure
of the image of $\Gamma$ in the Polish group $\Sym(\Gamma/\Lambda)$ is not discrete.
Then $\betti_n(\Gamma)=0$ for every $n\geq 0$.
\end{theorem}

\begin{proof}
We apply Corollary~\ref{cor: vanishing for n-step s-normal amenable subgroups}. 
We will be done by showing that for any $n$ and any $\omega\in \Gamma^n$, $\Lambda^\omega$ is infinite.
Assume otherwise that for some $n$ and some $\omega\in \Gamma^n$,
$\Lambda^\omega$ is finite.
Then for some $n'$ and $\omega'\in \Gamma^{n'}$,
$\Lambda^{\omega'}$ is the core of $\Lambda$, $\bigcap_{\gamma\in \Gamma} \Lambda^\gamma$.
That is, the identity element of the image of $\Gamma$ in $\Sym(\Gamma/\Lambda)$ could be expressed as
the intersection of finitely many open subgroups, contradicting the non-discreteness of this image.
\end{proof}


\end{document}